\documentclass[10 pt]{amsart}

\usepackage{latexsym, amsmath, amssymb, longtable, booktabs,amscd,microtype,booktabs,cases}
\usepackage[square]{natbib}
\usepackage{hyperref}
\bibpunct{[}{]}{,}{n}{}{;} 

\numberwithin{equation}{section}
\begin{document}

\newcommand\A{\mathbb{A}}
\newcommand\G{\mathbb{G}}
\newcommand\N{\mathbb{N}}
\newcommand\T{\mathbb{T}}
\newcommand\sO{\mathcal{O}}
\newcommand\sE{{\mathcal{E}}}
\newcommand\tE{{\mathbb{E}}}
\newcommand\sF{{\mathcal{F}}}
\newcommand\sG{{\mathcal{G}}}
\newcommand\GL{{\mathrm{GL}}}
\newcommand\HH{{\mathrm H}}
\newcommand\mM{{\mathrm M}}
\newcommand\fS{\mathfrak{S}}
\newcommand\fP{\mathfrak{P}}
\newcommand\fQ{\mathfrak{Q}}
\newcommand\Qbar{{\bar{\bf Q}}}
\newcommand\sQ{{\mathcal{Q}}}
\newcommand\sP{{\mathbb{P}}}
\newcommand{\Q}{{\bf Q}}
\newcommand{\tH}{\mathbb{H}}
\newcommand{\Z}{{\bf Z}}
\newcommand{\R}{{\bf R}}
\newcommand{\C}{{\bf C}}
\newcommand{\F}{\mathbb{F}}
\newcommand\gP{\mathfrak{P}}
\newcommand\Gal{{\mathrm {Gal}}}
\newcommand\SL{{\mathrm {SL}}}
\newcommand\Hom{{\mathrm {Hom}}}
\newcommand{\legendre}[2] {\left(\frac{#1}{#2}\right)}
\newcommand\iso{{\> \simeq \>}}

\newtheorem{thm}{Theorem}
\newtheorem{theorem}[thm]{Theorem}
\newtheorem{cor}[thm]{Corollary}
\newtheorem{conj}[thm]{Conjecture}
\newtheorem{prop}[thm]{Proposition}
\newtheorem{lemma}[thm]{Lemma}
\theoremstyle{definition}
\newtheorem{definition}[thm]{Definition}
\theoremstyle{remark}
\newtheorem{remark}[thm]{Remark}
\newtheorem{example}[thm]{Example}
\newtheorem{claim}[thm]{Claim}

\newtheorem{lem}[thm]{Lemma}

\theoremstyle{definition}
\newtheorem{dfn}{Definition}

\theoremstyle{remark}
\setlength{\abovedisplayskip}{2pt}
\setlength{\belowdisplayskip}{2pt}

\theoremstyle{remark}
\newtheorem*{fact}{Fact}
\makeatletter
\def\imod#1{\allowbreak\mkern10mu({\operator@font mod}\,\,#1)}
\makeatother

\title{The Eisenstein cycles as modular symbols}

\author{Debargha Banerjee}
\address{INDIAN INSTITUTE OF SCIENCE EDUCATION AND RESEARCH, PUNE, INDIA}
\author{Lo\"ic Merel}
\address{Univ Paris Diderot, Sorbonne Paris Cit\'e, Institut de Math\'ematiques de Jussieu-Paris Rive Gauche, UMR 7586, CNRS, Sorbonne Universit\'es, UPMC Univ Paris 06, F-75013, Paris, France}
\begin{abstract}
For any odd integer $N$, we explicitly write down the {\it Eisenstein cycles} in the first homology group of modular curves of level $N$ as linear combinations of Manin symbols.  These cycles are, by definition, those over which every integral of holomorphic differential forms vanish. Our result can be seen as an explicit version of the Manin-Drinfeld Theorem.  Our method is to characterize such Eisenstein cycles as eigenvectors for the Hecke operators. We make crucial use of expressions of Hecke actions on modular symbols and on auxiliary level $2$ structures.  \end{abstract}

\maketitle
 
\section{Introduction}
\label{Intro}
Let $N$ be a positive integer. Consider the principal congruence subgroup $\Gamma(N)$ which acts on the upper half-plane $\tH$ by homographies, and thus defines the modular curve $Y(N)=\Gamma(N) \backslash \tH$, compactified as $X(N)=Y(N) \cup \partial_N$, where $\partial_N=\Gamma(N) \backslash{\bf P}^{1}({\bf Q})$ is the set of cusps. 

The Manin-Drinfeld theorem \cite{MR0318157} \cite{MR0314846} asserts that the class of a divisor of degree zero supported on $\partial_N$ is torsion in the Jacobian $J(N)$ of the modular curve $X(N)$. The order of such divisor in $J(N)$ can made explicit by the use of Siegel units (\cite{MR648603}). We call {\it Eisenstein cycles} the elements $e\in \HH_1(X(N), \partial_N, {\bf R})$ such that $\int_{e}\omega=0$ for all $\omega\in \HH^0(X(N), \Omega^1)$. They constitute a real vector space $E(N)$. The Manin-Drinfeld theorem is reformulated by saying that $E(N)$ admits a ${\bf Q}$-rational structure in  $\HH_1(X(N), \partial_N, {\bf Q})$.
Our aim is to determine explicitly a basis for $E(N)$ in the following sense.
Let $\overline{\SL_2(\Z/N\Z)}=\SL_2(\Z/N\Z)/{\pm 1}=\pm\Gamma(N)\backslash\SL_2({\bf Z})$.
Let
$$\xi : \overline{\SL_2(\Z/N\Z)} \rightarrow \HH_1(X(N),\partial_N,{\bf Z})$$
be the map that takes the class of a matrix $g \in \mathrm{SL}_2({\bf Z})$ to 
the class in $\HH_1(X(N),\partial_N, \Z)$ of the image in 
$X(N)$ of the geodesic in $\tH \cup \sP^1(\Q)$ joining $g.0$ and $g.\infty$. It is surjective
\cite{MR0314846}. Hence we call $\xi(g)$ a {\it Manin generator}.
The Manin generators satisfy the {\it Manin relations}:
For all $g \in \Gamma(N) \backslash \SL_2({\bf Z})$, 
$\xi(g)+\xi(g S)=0$ and $\xi(g)+\xi(g U)+\xi(g U^2)=0$, where
 $T=\left(\begin{smallmatrix}
1 & 1\\
0 & 1\\
\end{smallmatrix}\right)$, $S=\left(\begin{smallmatrix}
0 & -1\\
1 & 0\\
\end{smallmatrix}\right)$ 
and  $U=ST=\left(\begin{smallmatrix}
0 & -1\\
1 & 1\\
\end{smallmatrix}\right)$.

Hence we wish to express the Eisenstein cycles as rational combinations of Manin generators. Because of the Manin relations, such a problem does not admit a unique solution. To eliminate this ambiguity, we introduce an auxiliary $\Gamma(2)$-structure. Thus we exibit a distinguished solution to our problem. The auxiliary level $2$-structure is reminiscent of the usual need to rigidify moduli problems by adding a level-structure, but we are not able to make an explicit connection. At any rate, our reliance on an auxiliary level $2$-structure limits our ability to treat the cases where $N$ is even. This is why we limit ourselves to the cases where the level $N$ is odd in the present article, except for a brief discussion in section 8, where we fully explain the role of the $\Gamma(2)$-structure.

Before we state our theorem, we need to explain to what extent our result does not seem to follow from the existing literature. 
We have a series of group isomorphisms:
\[
\HH_1(X(N), \partial_N, \Z) \simeq \Hom (\HH_1(Y(N),\Z),\Z)\simeq \Hom (\Gamma(N),\Z),
\]
where the first map comes from the intersection pairings and the inverse of the second map associates to $\gamma\in \Gamma(N)$ the class of the image in $Y(N)$ of a path from $z_{0}$ to $\gamma z_{0}$ in $\tH$ (where $z_{0}$ is any point in $\tH$). Hence finding the desired Eisenstein cycles amounts to make explicit certain group homomorphisms $\Gamma(N)\rightarrow {\bf Z}$.

In \cite{MR553997}, Mazur accomplished such a program when ${\Gamma(N)}$ is replaced by the  congruence subgroup $\Gamma_{0}(N)$, for $N$ odd prime. He calls the corresponding group homomorphism $\Gamma_{0}(N)\rightarrow {\bf Z}$ the {\it Dedekind-Rademacher homomorphism}, which is obtained as a period homomorphism for Eisenstein series. The method has been extended by Stevens  (\cite{MR670070}) to the more general modular curves $X(N)$ (without any restriction on parity for $N$). However, the group homomorphisms $\Gamma(N)\rightarrow {\bf R}$ exhibited by Stevens do not enable to write directly the Eisenstein cycles as linear combinations of Manin symbols.

Nevertheless what Mazur called the Dedekind-Rademacher homomorphism was used by one of us \cite{MR1405312} to find the desired expression of the (unique up to scalar in that case) Eisenstein cycle in the case where ${\Gamma(N)}$ is replaced by the  congruence subgroup $\Gamma_{0}(N)$, for $N$ odd prime. Already in that work, the introduction of an auxiliary $\Gamma(2)$-structure played a key role. By a similar method, one of us \cite{debargha}, and together with Krishnamoorty\cite{Pacific} treated the modular curves $X_{0}(p^{2})$ and $X_{0}(pq)$ respectively for $p$ and $q$ odd prime numbers, at the cost of significant additional technical difficulties.

Our method in the current paper does not rely on Dedekind-Rademacher type homomorphisms and yields more general results. We propose a formula for the Eisenstein elements and verify that they are eigenvectors for the Hecke operators. Our calculations to that extent depend crucially on the formulas for Hecke operators obtained by one of us in \cite{MR1316830}.

For $P= \left(\begin{smallmatrix}
{\bar x}\\
{\bar y}\\
\end{smallmatrix}\right)\in (\Z/N\Z)^2$, choose the representatives $x$, $y \in \Z$ of $\bar x$ and $\bar y$ respectively with $x-y$ odd. 
Define: 
\[
F(P)= -\frac{1}{4}[ \frac{(\cos (\frac{\pi x}{N})+\cos (\frac{\pi y}{N})}{(\cos (\frac{\pi x}{N})- \cos (\frac{\pi y}{N}) }] \in \R. 
\]
From the above expression, it is easy to see that for all $P \in (\Z/N \Z)^2$:
$F(P)=F(-P)$ and $F(P)+F(S P)=0$.

Let $\overline{F}$ be a function on $(\Z/N \Z)^2 / \{\pm 1\}$ obtained from $F$ by passing to quotients. Let
\[
\sE_P=\sum_{{\gamma} \in \overline{\SL_2(\Z/N\Z)}} {\bar F}(\gamma^{-1} P)\xi(\gamma)    = \sE_{-P}.
\]

\begin{theorem} 
\label{Eisensteinclass}
For $P \in (\Z/N\Z)^2$, the modular symbols $\sE_P$ satisfy the following properties:

1)
Suppose $l$ is an odd prime number and 
 $l \equiv 1 \pmod N$. Let $T_l$ be the Hecke operator at the prime $l$. The Eisenstein cycles
$\sE_P$ satisfy the equality:
\[
T_l(\sE_P) = (l + 1)\sE_P .
\]

2)
The  classes of $\sE_P$  lie in the kernel of $R$ and hence they are Eisenstein cycles. 

\end{theorem}
Part 2) of Theorem~\ref{Eisensteinclass} follows easily from part 1). They are proved by Proposition \ref{Heckeq} and Proposition \ref{kernel} respectively.

It remains to prove that the classes $\sE_{P}$ span the space $E(N)$ of Eisenstein cycles, when $P$ runs through elements of order $N$ of $P \in (\Z/N\Z)^2$.
We define the retraction $R$ as the map 
$\HH_1(X(N), \partial_N, \R)\rightarrow \HH_1(X(N), \R)$
characterized by $\int_{R(c)}\omega=\int_{c}\omega$ for all $\omega\in \HH^0(X(N), \Omega^1)$.
The kernel of $R$ coincides with $E(N)$. 
Let $\phi(N)$ be the order of $(\Z/N\Z)^*$.
Consider the group $D^+_N$  of even Dirichlet character $\chi$ modulo $N$.

Let 
 \[
L(\chi)=\frac{1}{2}\sum_{\mu\in(\Z/N\Z)^{*}}\frac{\chi(\mu)}{1-\cos(\frac{\pi\mu^{0}}{N})},
 \]
where $\mu^{0}$ is an odd representative of $\mu$ in $\Z$. It is essential for the next theorem that $L(\chi)$ is a non-zero algebraic number, as it can be expressed as the algebraic part of the value of a Dirichlet $L$-function.

\begin{thm}
\label{Retraction}
For all Manin-symbol $\xi(\gamma) \in \HH_1(X(N), \partial_N, \Z)$, we have the 
following retraction formula:
 \[
 R(\xi(\gamma))=\xi(\gamma)-\sum_{\alpha \in (\Z/N\Z)^*} \sum_{\chi \in D_N^{+}} 
  \frac{\chi(\alpha)}{2N\phi(N)L(\chi)}(\sE_{\alpha \gamma P_{\infty}}-\sE_{\alpha \gamma P_{0}}).
 \]
\end{thm}
Define the boundary map $\delta$ : $\HH_1(X(N), \partial_N, \Z)\rightarrow \Z[\partial_N]$ as follows. For
$r,s\in \sP^1(\Q)$, the image by $\delta$ of the geodesic of the upper half-plane joining the cusps $r$ to $s$ is
$[\pm\Gamma(N)r]- [\pm\Gamma(N)s]\in \Z[\partial_{N}]$.
For an element $x\in\HH_1(X(N), \partial_N, \R)$,
$\delta(x)=0$ implies that $R(x)=x$. 
Hence to prove Theorem \ref{Retraction}, it is enough to show that 
 \[
 \delta(\xi(\gamma))=\delta(\sum_{\alpha \in (\Z/N\Z)^*} \sum_{\chi \in D_N^{+}} 
  \frac{\chi(\alpha)}{ 2N \phi(N)L(\chi)}(\sE_{\alpha \gamma P_{\infty}}-\sE_{\alpha \gamma P_{0}})).
 \]
We prove this statement in Proposition~\ref{done}. From the above theorem, 
we can easily conclude that:
\begin{cor}
The kernel of $R$ is spanned by the classes of $\sE_P$
for  $P \in (\Z/N \Z)^2$ with $P$ of order $N$. 
\end{cor}

We can offer a straightforward application of Theorem \ref{Retraction}. It enables to write explicitly $L(f,1)$, for $f$ cuspidal modular form of weight $2$ for $
\Gamma(N)$ as a linear combination of periods of $f$. Indeed, one has
\[
\int_{0}^{\infty}f(z)dz=\int_{R(\xi({\rm Id}))}\omega_{f},
\]
where $\omega_{f}$ is the Pullback to $X(N)$ of the differential form $f(z)dz$.

We envisage as well applications similar to those of \cite{MR1405312} for the structure of Hecke algebras completed at Eisenstein primes. There is a possibility that the computation carried out 
in this paper is related to the Eisenstein classes considered in  [Chapter 8 ~\cite{Venkatesh}]. This 
computation should be useful to answer the questions raised in the book. Finally, we have not tried to put our result in perspective with the Ramanujan sums considered in ~\cite{MR983619} by Murty and Ramakrishnan.

 \section{Acknowledgement}
Both authors wish to thank IMSC, Chennai for providing excellent work environment. The second author wishes to thank the program IRSES Moduli for funding his stay in IMSC. The first author 
was partially supported by the SERB grant YSS/2015/001491.
 \section{Modular symbols and retraction maps}
 \label{MSretract}
 
We have a long exact sequence of relative homology:
\[
0  \rightarrow  \HH_1(X(N),\Z) \rightarrow \HH_1(X(N), \partial_N,\Z) \xrightarrow {\delta} {\bf Z}[\partial_N]\rightarrow \Z \rightarrow 0.
\]
The first non-trivial map is the canonical injection. For $r,s\in \sP^1(\Q)$, {\it the modular symbol} $\{r,s\}$ is the class in $\HH_1(X(N),\partial_N,\Z)$ 
of a continuous path in in $\tH \cup \sP^1(\Q)$ joining $r$ to $s$. By definition, its image by the boundary map $\delta$ is 
$[\Gamma(N)r]- [\Gamma(N)s]\in \Z[\partial_{N}]$. 
The last non-trivial map of the long exact sequence is the sum of the coefficients. 

Let $\xi : \mathrm{SL}_2(\Z) \rightarrow \HH_1(X(N),\partial_N,\Z)$ 
be the map that takes the matrix $g \in \mathrm{SL}_2(\Z)$ to the modular symbol $\{g.0,g.\infty\}$.
The map $\xi$ is surjective (Manin ~\cite{MR0314846}).
We note that the class of $\xi(\gamma)=[\gamma]$ depends only on the class of $\gamma$ in 
$\overline{\SL_2(\Z/N\Z)}:=\SL_2(\Z/N\Z)/{\pm Id}$.

We have an isomorphism of real vector spaces $\HH_1(X(N),\Z)\otimes \R = \Hom_{\C}(\mathrm{H}^0(X(N),\Omega^1),\C)$
given by 
\[
c \rightarrow \{\omega \rightarrow \int_c \omega\}.
\]
Consider the group homomorphism $\HH_1(X(N),\partial_N,\Z) \rightarrow  \Hom_{\C}(\HH^0(X(N),\Omega^1),\C)$
given by 
\[
c \rightarrow \{\omega \rightarrow \int_c \omega\}.
\]
The above homomorphism defines a retraction map $R:$
\[
\HH_1(X(N), \partial_N, \Z) \rightarrow \HH_1(X(N),  \R). 
\]

The Manin-Drinfeld theorem is equivalent  \cite{MR1363488} to the fact that the image of $R$ is contained in $\HH_1(X(N), \Q) \subset \HH_1(X(N), \R)$, {\it i.e} $R$ splits the long exact sequence after extending the scalars to $\Q$.

\section{Hecke operators acting on modular symbols}
\label{Hecke}
In this section, we recall the action of the Hecke operators on the space of modular symbols  [\cite{MR1322319}, \cite{MR1136594}]. Suppose $m$ is a positive integer congruent to $1$ 
modulo $N$. Let $A_m$ be the set of matrices in $M_2(\Z)$ of determinant $m$ 
and $A_{m,N}$ be the set of matrices in $A_m$ which are congruent to identity  modulo $N$. The 
congruence subgroup $\Gamma(N )$ acts on the right on $A_{m,N}$. Let $R$  be a system of 
representatives of $\Gamma(N )\backslash  A_{m,N}$. When $m$ is congruent to $1$ modulo $N$, the Hecke correspondence $T_m$
on $X(N )$ is defined by 
\[
\Gamma(N)z \rightarrow \sum_{r \in R} \Gamma(N) rz.
\]
This action does not  dépend on the choice of $R$. It fixes the set of cusps of $X(N )$ pointwise. 
Thus, by transport of structure, the Hecke correspondence $T_m$  defines an endomorphism $T_m$ on $\HH_1(X(N),\partial_N, \Z)$ and we have, for $\alpha$, $\beta\in\sP^1(\Q)$, $T_{m}(\{\alpha,\beta\})=\sum_{r \in R} \{r\alpha,r\beta\}$

\begin{definition}
[Condition $(C_m)$]
 An element
$\sum_M u_M M \in \Z[A_m]$
 satisfies the following condition $C_m$  \cite{MR1322319} if  we have the following equality in  $\C[\sP^1(\Q)]$:
\[
\sum_{M \in [K]} u_M M((\infty)-M(0))=(\infty)-(0)
\]
 for all classes $K \in M_2(\Z)_m \slash {SL}_2(\Z)$.
\end{definition}

 Let $\sum_M u_M M \in \Z[A_m]$ satisfies the condition $C_m$.
The action of the Hecke operator $T_m$ on Manin symbol  $\xi(g) \in \overline{\SL_2(\Z/N\Z)}$ is given by the following  formula \cite{MR1322319}:
\[
T_m (\xi(g)) =\sum_M u_M  \xi(gM),
\]
which is meaningful since the elements of $A_m$ are of determinant $1$ modulo $N$. Similar formulas hold when $m$ is any interger ({\it i.e.} not necessarily congruent to $1$ modulo $N$) but one needs to make a choice in the definition of $T_{m}$.

We turn to a particular family $\Z[A_m]$ that satisfies condition $C_{m}$, provided we impose that $m$ is an odd integer, which is our assumption for the rest of this section.

Consider the following two sets of matrices:
\[
U_m = \{
\left(\begin{smallmatrix}
x & -y\\
y' & x'\\
\end{smallmatrix}\right)
\in M_2(\Z)/ xx' + yy' = m, x \text{ and }  x' \text{ odd},
y \text{ and } y' \text{ even } x > |y|, x' > |y'|\}
\}
\]
and 
\[
V_m = 
\{
\left(\begin{smallmatrix}
x & -y\\
y' & x'\\
\end{smallmatrix}\right)
\in M_2(\Z)/ xx' + yy' = m, x \text{ and }  x' \text{ odd },
y \text{ and } y' \text{ even } x > |y|, x' > |y'|\}
\}.
\]
 For all $m$, we define 
\[
\theta_m=\sum_m U_m-\sum_m V_m=\sum_m u_m^{\theta} M.
\]
The element $\theta_m$ satisfies the condition $(C_m)$ \cite{MR1322319}.  Taken together, these elements satisfy the properties of Hecke operators in $\Z[M_2(\Z)]$ : $\theta_{m}\theta_{m'}=\theta_{mm'}$, provided $m$ and $m'$ are odd and coprime, and the usual recursive formula when $m$ is the power of a prime number \cite{MR1316830}.

Let us denote 
by $\overline{M_2(\Z)}$, the set of matrices in $M_2(\Z)$ modulo multiplication by $\pm 1$. 
For $M=\left(\begin{smallmatrix}
a & b\\
c & d\\
\end{smallmatrix}\right)$, let 
 $\widetilde{M}=\left(\begin{smallmatrix}
d & -b\\
-c & a\\
\end{smallmatrix}\right)$.

Let $H=\left(\begin{smallmatrix}
1 & 1\\
-1 & 1\\
\end{smallmatrix}\right)$ and $c=\left(\begin{smallmatrix}
-1 & 0\\
0 & 1\\
\end{smallmatrix}\right)$.

\begin{prop}\label{Hecke}
[\cite{MR1316830}, Proposition 5, 6]
We have the following relations in
$\Z[\overline{M_2}(\Z)]$:
\begin{itemize}
\item 
$\theta_m c=c \theta_m$
\item 
$\theta_m H=H \widetilde{\theta_m}+ ([1]+[S])\sum_{M \in V_m} [MH-H\widetilde{M}]$, 
\end{itemize}
where 
\[
\widetilde{\theta_m}=\sum_M u_M^{\theta} [\tilde{M}].
\]
\end{prop}

We will not use the explicit forms of the Diophantine sets $U_m$ and $V_m$ but 
we will make an use of the condition $C_m$ and the properties of $\theta_m$
stated in the proposition. The latter one is essential, but we only need to remember that $\theta_m H-H \widetilde{\theta_m}$ is a left multiple of $[1]+[S]$.
It will be useful to note that the support of $\theta_{m}$ is contained in the set of matrices congruent to the identity modulo $2$.

\section{Another expression of the function $F$}
Let $\overline{B_1}:\R \rightarrow  \R$ be the first Bernoulli function. This is a  periodic with period $1$ and for $x \in (0, 1)$, $\overline{B_1}$ is defined by 
$\overline{B_1}(x) = x -\frac{1}{2}$ and $B_1( 0) = 0$.  Let $e(x)=e^{2i\pi x}$.
We still denote by
$\overline{B_1}$ and $e$ functions defined on
the quotient  $\R \slash \Z$. 
Recall that $H=\left(\begin{smallmatrix}
1 & 1\\
-1 & 1\\
\end{smallmatrix}\right)$.
 For $P \in (\Z/N\Z )^2$, we denote by $P^0$ the unique  ́element 
of  $(\Z/2N\Z)^2$ which coincides with $P$ modulo $N$ and to $(1,1)$ modulo $2$. Recall that $H=\left(\begin{smallmatrix}
1 & 1\\
-1 & 1\\
\end{smallmatrix}\right)$.
\begin{prop}
\label{Alternative}
For $P \in (\Z/N\Z)^2$,  we have 
\[
  F (P ) =\sum_{s=(s_1,s_2) \in (\Z/2N\Z)^2} e(\frac{s(HP)^0}{2N}) \overline{B_1}(\frac{s_1}{2N}) \overline{B_1}(\frac{s_2}{2N}).
\]
\end{prop}
\begin{proof}
Let $(x,y) \in \Z^2$ be a representative such that $(x-y)$ is odd. 
For $(u,v)=(HP)^0=(x+y,x-y)+2N \Z^2 \in (\Z \slash 2 N \Z)^2$, 
Consider the sum $\sum_{k=0}^{2N-1} e(-\frac{ku}{2N})\overline{B_1}(\frac{k}{2N})$. 
We have 
\begin{eqnarray*}
\sum_{k=0}^{2N-1} e(-\frac{ku}{2N})\overline{B_1}(\frac{k}{2N})
&= & \sum_{k=1}^{2N-1} e(-\frac{ku}{2N})[\frac{k}{2N}-\frac{1}{2}]\\
&= & \frac{1}{2N} \sum_{k=0}^{2N-1} e(-\frac{ku}{2N}) k-\frac{1}{2}\sum_{k=1}^{2N-1} e(-\frac{ku}{2N})\\
&= & \frac{1}{e(-\frac{u}{2N})-1}+\frac{1}{2}\\
&= & \frac{e(-\frac{u}{2N})+1}{2(e(-\frac{u}{2N})-1)}\\
&= &-\frac{i}{2} \cot(-\frac{ \pi u}{2N}).
\end{eqnarray*}
The last equality follows from $1-e(t)=-2i \sin(\pi t) e(\frac{t}{2})$  
and $1+e(t)=2 \cos(\pi t) e(\frac{t}{2})$
for all $t \in \R$.
 We deduce the following 
equality:
\begin{eqnarray*}
\sum_{k_1=0}^{2N-1} \sum_{k_2=0}^{2N-1} e (- \frac{k_1u+k_2 v}{2N}) 
\overline{B_1}(\frac{k_1}{2N})\overline{B_1}(\frac{k_2}{2N})
&= & \frac{e(-\frac{u}{2N})+1}{2(e(-\frac{u}{2N})-1)}\frac{e(-\frac{v}{2N})+1}{2(e(-\frac{v}{2N})-1)}\\
&= & \frac{1}{4}  [\frac{e(-\frac{u}{2N})+1}{e(-\frac{u}{2N})-1}][\frac{e(-\frac{v}{2N})+1}{e(-\frac{v}{2N})-1}]\\
&= &- \frac{1}{4} \cot ( \frac{\pi(x+y)}{2N}) \cot( \frac{\pi(x-y)}{2N}) \\
&= &-\frac{1}{4}[ \frac{\cos (\frac{\pi x}{N})+ \cos (\frac{\pi y}{N})}{ \cos (\frac{\pi x}{N})- \cos (\frac{\pi y}{N}) }]. \\
\end{eqnarray*} 
\end{proof}

\section{The element $\theta_m$ and the function $F$}
We will prove a main result in this section. For $(x,y) \in \R^2$, 
set $f(x,y) =: \overline{B}_1 (x)\overline{B}_1 (y)$.
Consider the Fourier series development of the Bernoulli number $\overline{B}_1$. 
For $x \in \R$, recall the Fourier expansion of $\overline{B_1}$:
\[
-\frac{1}{2 \pi i}\sum_{n \neq 0}^{\infty} \frac{1}{n} e(nx)=\overline{B_1}(x).
\]
Hence, we have
\[
f(x,y)= \overline{B_1}(x)\overline{B_1}(y)=\frac{1}{(2 \pi i)^2}\
\sum_{n \neq 0}^{\infty} \sum_{m \neq 0}^{\infty} \frac{1}{n m} e(nx+my).
\]

\begin{prop}
\label{Heckeaction}
 Let $l$ be a prime number. Wuppose $\sum_L u_L [L] \in \Z[A_l]$  satisfies 
 the condition $C_l$ and 
 $\sum_L u_L [Lc] = \sum_L u_L [cL]$
 in $[\overline{M}_2 (\Z)]$.
 We then have 
\[
\sum_L u_L f(sL)=lf(s)+f(ls)
\]
for all $s \in \Z^2$.
\end{prop}
\begin{proof}
Consider the functions
\[
f_g : (x, y) \rightarrow -4 \pi^2 \sum_L u_L f ((x, y)L) 
\]
and 
\[
f_d :(x, y) \rightarrow -4 \pi^2 (l f (x, y)+f (lx, ly)).  
\]
They are periodic functions with period $1$ in both variables.

To prove $f_g=f_d$, it is enough to prove the equality of Fourier coefficients 
\[
c_{n,m} (f_g )=c_{n,m} (f_d ) 
\]
for $(n, m) \in \Z^2$.

Now, 
\[
f_d(x,y)=-4 \pi^2 (l f (x, y)+f (lx, ly))=[\sum_{n \neq 0}^{\infty} \sum_{m \neq 0}^{\infty} \frac{l}{n m} e(nx+my)+\sum_{n \neq 0}^{\infty} \sum_{m \neq 0}^{\infty} \frac{1}{n m} e(nlx+mly)].
\]

We first consider the coefficients $c_{n,m}(f_d)$. We first note that 
$c_{n,0}(f_d) = c_{0,m} (f_d) = 0$. Moreover we have
\begin{equation*}
c_{n,m} (f_d)  =
\begin{cases}
\frac{l}{nm} & \text{If $(n, m) \notin (l\Z)^2$ and $nm \neq 0$,} \\
\frac{l}{nm} + \frac{l^2}{nm} &  \text{If $(n, m) \in (l\Z)^2$ and $nm \neq 0$.} \\
\end{cases}
\end{equation*}

We now examine the coefficients of $c_{n,m}(f_g)$. We note that  
$\sum_L u_L Lc = c \sum_L u_L L$ 
and $f_g$ is an odd function in both coordinate. We deduce that 
$c_{n,0}(f_g ) = c_{0,m}(f_g)=0$.
Suppose now that $nm \neq 0$ and $L=\left(\begin{smallmatrix}
a & b\\
c & d\\
\end{smallmatrix}\right) \in A_l$. We deduce that 
\[
f((x,y)L)= \sum_{(n,m) \in \Z^2, nm \neq 0} \frac{e((ax+cy)n+(bx+dy)m)}{nm}
=\sum_{(n,m) \in \Z^2, nm \neq 0} \frac{e((an+bm)x+(cn+dm)y)}{nm}.
\]
For $(n', m')=(an+bm, cn+dm) \in \Z^2 L^t$,  we  then have  
\[
f((x,y)L)= \sum_{(n,m) \in \Z^2 L^t, (dn'-bm'),(-cn'+am')\neq 0} \frac{l^2e(xn'+ym')}{(dn'-bm')(-cn'+am')}.
\]
We deduce that 
\[
c_{m,n}(f_g)=\sum_{L} \frac{u_Ll^2}{(dn-bm)(-cn+am)}
\]
with the summation over all matrices $L$ in $A_l$ such that $(n,m) \in \Z^2 L^t$ and 
$(dn-bm)(-cn+am) \neq 0$. We deduce the following relation for all $\left(\begin{smallmatrix}
a & b\\
c & d\\
\end{smallmatrix}\right) \in A_l$, 
\[
\frac{1}{(dn-bm)(-cn+am)}=\frac{1}{ad-bc}(\frac{1}{m(n-\frac{b}{d}m)}-\frac{1}{m(n-\frac{a}{c}m)})
=\frac{1}{l}(\frac{1}{m(n-L(0)m)}-\frac{1}{m(n-L (\infty)m)}).
\]
We also deduce that 
\[
c_{n,m}(f_g)=\sum_{\alpha \in A_l \slash \SL_2(\Z)} \sum_L l u_L(\frac{1}{m(n-L(0)m)}-\frac{1}{m(n-L (\infty)m)}),
\]
with the second summation over all matrices $L$ in $A_l$ such that $L \in \alpha$, $(n,m) \in \Z^2 L^t$ and $(dn-bm)(-cn+am) \neq 0$. The set $\Z^2 L^t$ does not depend on the class of $L $  in $\SL_2(\Z) \Z^2 =\Z^2$.

Now, suppose that $\Z^2 L^t=\Z^2 \alpha^t$. We also assume that $dn-bm=0$ if and only 
if $L0=\frac{n}{m}$ and $-cn+am=0$ if and only of $L(\infty)=\frac{n}{m}$.
We deduce that the condition $C_l$ gives the following equality:
\[
c_{n,m}(f_g)=\sum_{\alpha \in A_l \slash \SL_2(\Z)} \frac{l}{nm}.
\]
The set of classes in $A_l/\SL_2(\Z)$ are in bijection with the subgroups of finite 
index $l$ of $\Z^2$.
If $\alpha$ and $\alpha'$ are two distinct elements of $A_l/ \SL_2(\Z)$, we have
\[
\Z^2 \alpha^t \cap \Z^2 (\alpha')^t=l^2. 
\]
Also, we have $|A_l \backslash \SL_2(\Z)|=l+1$. Every element of $\Z^2 -l\Z^2$ corresponds 
 to an unique subgroup of index $l$  in $\Z^2$ and an unique class of $\Z^2 \alpha^t$ with 
 $\alpha \in A_l \backslash \SL_2(\Z)$. We now calculate $c_{n,m}(f_g)$. We have 
 \[
 c_{n,m}(f_g)=\frac{l}{nm}
 \]
 if $(n,m) \notin l \Z^2$ and 
 \[
 c_{n,m}(f_g)=\sum_{\alpha \in A_l \slash \SL_2(\Z)} \frac{l}{nm}=  (l+1)\frac{l}{nm}
 \]
 if $(n,m) \in l \Z^2$. We thus prove the theorem. 
\end{proof}
\begin{remark}
Let $m$ be a positive integer. Suppose $\sum_M u_M M \in \Z[A_m]$ satisfies the condition 
$(C_m)$ and such that $\sum_M u_M Mc=c \sum_M u_M M$. For $s \in \Z^2$, by imitating 
the above method we can prove that 
\[
\sum_M u_M f(sM)=\sum_{d \mid m} \frac{m}{d} f(ds). 
\]

We note that the condition  $\sum_M u_M [Mc]=\sum_M u_M [cM]$ is not essential. In fact, if $\sum_M u_M [M]$ 
satisfies the condition $(C_m)$ and not necessarily commute with $c$.  Then the element 
\[
\frac{1}{2} \sum_M u_M ([M] +[cMc]) \in \Z[ \frac{1}{2}] [M_2(\Z)]
\]
satisfy the condition $(C_m)$ and it commutes with $c$. 
\end{remark}

For 
$P=\left(\begin{smallmatrix}
x\\
y \\
\end{smallmatrix}\right) \in (\Z/2N\Z)^2$,
consider the function 
\[
\widehat{f}(P)= \sum_{s \in (\Z/2N \Z)^2} e(\frac{ s P}{2N}) f(\frac{s}{2N}).
\]
 
Observe that $\widehat{f}(P)+\widehat{f}(S P)=0$ again since 
\[
\overline{B}_1(-\frac{s}{2N})=-\overline{B}_1(\frac{s}{2N}).
\]
\begin{prop}
\label{ccommute}
Let $l$ be an odd prime number and let $\sum_L u_L [L] \in \Z[A_l]$ satisfying the 
condition $C_l$ such that 
$
\sum_L u_L [L c]=c \sum_L u_L  [cL]
$
in $\Z [M_2(\Z)]$. For $P \in (\Z/N\Z)^2$, we have 
\[
\sum_L u_L \widehat{f}(\tilde{L} P)=l \widehat{f}(lP)+\widehat{f}(P). 
\]
\end{prop}
\begin{proof}
 For $P \in  (\Z/N\Z)^2$, we have 
\[
\sum_L u_L \widehat{f}( \widetilde{L}P)=\sum_L u_L \sum_{s \in (\Z/2N\Z)^2} e(\frac{s \widetilde{L}P}{2N})
f(\frac{s}{2N}).
\]
By abuse of notation, we consider $f$ as a function on $(\Z/2N\Z)$ defined by passing to the quotients.

For $N$ co-prime to $l$, we make a change of variable $tL=s$ on $(\Z/N\Z)^2$. 
By using Prop \ref{Heckeaction} and the relation 
$L  \tilde{L}=\left(\begin{smallmatrix}
l & 0\\
0 & l\\
\end{smallmatrix}\right)$, we have the following equality:
\begin{eqnarray*}
\sum_L u_L \widehat{f} (\tilde{L}P) &= & \sum_L u_L \sum_{s \in (\Z/2N\Z)^2} e(\frac{s \widetilde{L} P}{2N})f(\frac{s}{2N})\\
&= & \sum_L u_L \sum_{t \in (\Z/2N\Z)^2} e(\frac{t L\widetilde{L} P}{2N})f(\frac{t L}{2N})\\
  &= & \sum_{s \in (\Z/2N\Z)^2} e(\frac{l t P}{2N}) \sum_L u_L  f(\frac{t L}{2N})\\
            &= & \sum_{t \in (\Z/2N\Z)^2} e(\frac{t l P}{2N}) (l f(\frac{t}{2N})+f(\frac{lt}{2N}))\\
             &= & l \widehat{f}(lP)+\widehat{f}(P).
\end{eqnarray*}
\end{proof}
Theorem~\ref{Eisensteinclass} follows directly from the following proposition.
\begin{prop}
Suppose $l$ is an odd prime number and $P \in (\Z/N\Z)^2$. Then, we have 
\[
\sum_L u_L^{\theta} F(L P)= F(P)+l F (lP). 
\]
\end{prop}
\begin{proof}
From Proposition~\ref{Alternative}, we have:
\[
F(P)=\widehat{f} ((HP)^0).
\]
Thus
\[
\sum_L u_L^{\theta} F(L P)=\sum_L  u_L^{\theta} \widehat{f} ((HLP)^0).
\]
We use the formula expressed in proposition ~\ref{Hecke}:
\[
\theta_l H=H \widetilde{\theta_l}+ ([1]+[S])\sum_{M \in V_l} [MH-H\widetilde{M}]=H \widetilde{\theta_l}+ ([1]+[S])\sum_{M}a_{M}[M],
\]
where $M$ runs over matrices congruent to the identity modulo $2$.
Thus we obtain:

\begin{eqnarray*}
\sum_L u_L^{\theta} F(LP) &= & \sum_L u_L^{\theta} \widehat{f} ((H L P)^0) \\
         &= & \sum_L u_L^{\theta} \widehat{f} (\widetilde{L} H(P)^0)+\sum_{M}a_{M} (\widehat{f} (MP)^0)+\widehat{f} (SMP)^0)).
\end{eqnarray*}
We use the fact that $(SMP)^{0}=S(MP)^{0}$ and the antiinvariance of $f$ under $S$. Hence the last term vanishes. We can pursue the calculation using the property of $\widetilde{f}$ established in the previous proposition and the fact that, for any 
$P \in (\Z/N\Z)^2$ and $L \in R$ congruent to the identity modulo $2$,  we have $L P^0=(LP)^0$

\begin{eqnarray*}
\sum_L u_L^{\theta} F(LP) &= &\sum_L u_L^{\theta} \widehat{f}(\tilde{L}H P)^0)\\
         &= & \sum_L u_L^{\theta} \widehat{f} (\tilde{L}(H P)^0)\\
           &= & \widehat{f} ((H P)^0)+l \widehat{f} (l(H P)^0)\\
             &= & F(P)+lF(l P).
\end{eqnarray*}
\end{proof}

\section{Eisenstein eigenvectors}
We now prove  the  Theorem~\ref{Eisensteinclass}. 
\begin{prop}
\label{Heckeq}
Suppose $P \in (\Z /N\Z)^2$ and $l$ be an odd prime number congruent to $1$ modulo $N$. 
On $\HH_1(X(N), \partial_N, \Z)$, we have 
\[
T_l(\sE_P)= (l+1) \sE_P.
\]
\end{prop}
\begin{proof}
We use the propositions in the previous section. In particular, we use the following relation:
\[
\sum_L u_L^{\theta} F(LP)=F(P)+l F(lP). 
\]
If $l \equiv 1  \pmod N$, the reductions of the matrices in the support of $\theta_l$ are of determinant 
$1$ modulo $N$.   We have
\[
\sum_L u_L^{\theta} F(LP)=F(P)+lF(P)=(1+l)F(P). 
\]
We deduce the equality, 
\[
T_l(\sE_P)=\sum_L u_L^{\theta} \sum_{\gamma \in \overline{\SL_2(\Z/N\Z)}} \overline{F}(\gamma^{-1}P)[\gamma L] 
\]
\[
=\sum_{\gamma \in \overline{\SL_2(\Z/N\Z)}}\sum_L u_L^{\theta} \overline{F}(L\gamma^{-1}P)[\gamma ]
=(l+1)\sum_{\overline{\SL_2(\Z/N\Z)}}\overline{F}(\gamma^{-1}P)[\gamma]
=(l+1)\sE_P
.
\]
In the second step, we use the change of variable $\gamma L=\gamma'$. 
\end{proof}
We now prove the second part of Theorem~\ref{Eisensteinclass}. 
\begin{prop}
\label{kernel}
The classes of $\sE_P$ lie in the kernel of the map $R$. 
\end{prop}
\begin{proof}
Recall, $T_l$ is the Hecke operator on  $\HH^0(X(N), \Omega^1)$ deduced from the correspondences $T_l$ of degree $l+1$. Since the operators $T_l -(1+l)$ is surjective 
on $\HH^0(X(N), \Omega^1)$ \cite{MR0318157}, any  $\tilde{\omega} \in \HH^0(X(N), \Omega^1)$
can be written in of the form $\tilde{\omega}=(T_l-(1+l)) \omega$.  We then have 
\[
\int_{\sE_P}\tilde{\omega}=\int_{\sE_P}(T_l-(1+l)) \omega=\int_{\sE_P}T_l \omega-(l+1)\int_{\sE_P} \omega
\]
\[
=\int_{T_l\sE_P}\omega-(l+1)\int_{\sE_P} \omega=(l+1)\int_{\sE_P}\omega-(l+1)\int_{\sE_P} \omega=0
\]
We deduce that $\sE_P$ lies in the kernel of $R$. We have proved Theorem~\ref{Eisensteinclass}. 
\end{proof}

\section{Finer Eisenstein eigenvectors in mixed homology groups}

This section is not stricly useful for the main results of the article. But we hope to enlighten the reader about our constructions.

Let $M$ be an even integer. 
Consider the morphism of Riemann surfaces $\pi$ : $X(M)\rightarrow X(2)$. 
Recall that the modular curve $X(2)$ has three cusps: the classes $\Gamma(2)0$, $\Gamma(2)1$ and $\Gamma(2)\infty$ of $0$, $1$ and $\infty$ $\in{\bf P}^{1}({\bf Q})$. 
Consider the following partition of the cusps of $X(M)$ into $P_{+}$ and $P_{-}$, where $P_{+}=\pi^{-1}(\{\Gamma(2)0,\Gamma(2)\infty\})$ and $P_{-}=\pi^{-1}(\{\Gamma(2)1\})$. 

Thus we can consider the mixed homology groups 
$\HH_{0}=\HH_{1}(X(M)-P_{+},P_{-};\Z)$ and $\HH^{0}=\HH_{1}(X(M)-P_{-},P_{+};\Z)$. The intersection pairing $\bullet$ provides a $\Z$-valued perfect pairing between the latter two groups. 

More precisely, consider the map $\xi^{0}$ : $\Z[\pm\Gamma(M)\backslash \Gamma(2)]\rightarrow \HH_{0}$ that associates to $\gamma\in \Gamma(2)$ the class in $\HH^{0}$ of the image in $X(M)$ of the geodesic path in $\tH$ from $\gamma0$ to $\gamma\infty$. 
It has a counterpart 
$\xi_{0}$ : $\Z[\pm\Gamma(M)\backslash \Gamma(2)]\rightarrow \HH^{0}$ that associates to $\gamma\in \Gamma(2)$ the class in $\HH_{0}$ of the image in $X(M)$ of the geodesic path in $\tH$ from $\gamma(-1)$ to $\gamma1$. 

In \cite{MR1405312}, one of us proved that both  $\xi^{0}$ and $\xi_{0}$ are group isomorphisms. Moreover one has $\xi^{0}([x])\bullet\xi_{0}([y])=0$ if $x\ne y$ and $\xi^{0}([x])\bullet\xi_{0}([y])=1$ if $x=y$ ($x$, $y\in \pm\Gamma(M)\backslash \Gamma(2)$).

For $m$ odd integer, the Hecke correspondence $T_{m}$ leaves stable the sets of cusps $P_{+}$ and $P_{-}$. Hence it defines endomorphisms of the groups $\HH_{0}$ and $\HH^{0}$. 

More precisely, when $m$ is congruent to $1$ modulo $M$, the Hecke action if given thus on the bases of $\HH_{0}$ and $\HH^{0}$ \cite{MR1316830}:
\[
T_{m}\xi^{0}(\gamma)=\xi^{0}(\gamma\theta_{m})
\]
and 
\[
T_{m}\xi_{0}(\gamma)=\xi_{0}(\gamma\widetilde{\theta_{m}}).
\]
Similar formulas hold even when $m$ is any odd, positive integer \cite{MR1316830}.

Suppose now that $M=2N$. One has canonical identifications $\pm\Gamma(M)\backslash\Gamma(2)\simeq\pm\Gamma(N)\backslash\SL_2(\Z) \simeq\overline{\SL_2(\Z/N\Z)}$.

On the other hand, we have a canonical map $\lambda$ obtained by composing the maps: 
\[
\HH^{0}=\HH_{1}(X(M)-P_{-},P_{+};\Z)\rightarrow\HH_{1}(X(M),P_{+};\Z)\rightarrow\HH_{1}(X(M),\partial_{M};\Z)\rightarrow \HH_{1}(X(N),\partial_{N};\Z).
\] 
The last map is deduced from the obvious degeneracy map $X(M)\rightarrow X(N)$. 
The map $\lambda$ has the following two properties:
For $\gamma\in \Gamma(2)$, one has $\lambda(\xi^{0}(\gamma))=\xi(\gamma)$. Therefore $\lambda$ is surjective.

It respects the Hecke operators and therefore the image by $\lambda$ of an Eisenstein class is an Eisenstein class.

Our main innovation in the current paper, is to lift the Eisenstein classes of $\HH_{1}(X(N),\partial_{N};\Z)$ to certain Eisenstein classes of $\HH^{0}=\HH_{1}(X(M)-P_{-},P_{+};\Z)$, which admits a canonical basis (via $\xi^{0}$), and therefore the latter classes admit a canonical expression in terms of Manin symbols, which via $\lambda$ gives the distinguished expressions for Eisenstein classes obtained in this article. Evidently, our method presents difficulties to find Eisenstein classes of even level. Nevertheless, some non trivial things can be said in that context.

Hence the classes 
\[
\sE^{0}_P=\sum_{{\gamma} \in \overline{\SL_2(\Z/N\Z)}} {\bar F}(\gamma^{-1} P)\xi^{0}(\gamma)    = \sE_{-P}
\]
are Eisenstein classes of $\HH^{0}=\HH_{1}(X(M)-P_{-},P_{+};\Z)$, which can be proved in the same way than 
Theorem~\ref{Eisensteinclass}.

Via the canonical map $\HH^{0}=\HH_{1}(X(M)-P_{-},P_{+};\Z)\rightarrow \HH_{1}(X(M),\partial_{M};\Z)$ one obtains Eisenstein classes in $ \HH_{1}(X(M),\partial_{M};\Z)$. Denote by $\sE'_{P}$ the image of $\sE^{0}_P$. Thus we are already beyond the scope of Theorem~\ref{Eisensteinclass}. However, the Hecke operator $U_{2}$ acts on $ \HH_{1}(X(M),\partial_{M};\Z)$.

It's unclear to us whether $E(M)$ is spanned, as a $\Z[U_{2}]$-module by the classes $\sE'_{P}$.

\section{Boundaries of Eisenstein classes}

Let  $C_N$ be the set of points of $(\Z/N \Z)^2$ of order $N$ modulo multiplication
by $\{\pm 1\}$. 
Denote by $\phi$ the map that associate the representative
$[\frac{u}{v}] \in  \sP^1(\Q)$  of  $\Gamma(N )\backslash \sP^1(\Q)$
the class of $(u,v)$ in $C_N$. 
We identify the sets $\partial_N$ with $C_N$ using the map $\phi:\Gamma(N) \backslash\sP^1(\Q) \rightarrow C_N$ \cite{MR2112196}. Let  $(P)$ be an image in $C_N$  of a cusp in the above bijection. 
For an element  $P$ of order $N$ in $(\Z/N\Z)^2$, 
 denote by $(P)$ the cusp corresponding to the element $\overline{P}$  formed by the above bijection.
We calculate the boundary $\delta(\sE_P) \in \Z[C_N]$ of the Eisenstein class $\sE_P$.

\begin{thm}
\label{boundary} One has
\[
\delta(\sE_P)=2N[\sum_{\mu \in (\Z/N\Z)^*} (F(\mu P_{\infty}) +\frac{1}{4})(\mu P)]-\frac{N}{2} [\sum_{Q \in C_{N}} (Q)].
\]
\end{thm}

\begin{proof}
Since $S(P_{\infty})=-P_0$, we deduce that the boundary of $[\gamma]$ is equal to 
 \[
 \delta([\gamma])=(\gamma \infty)-(\gamma 0)=(\gamma P_{\infty})-(\gamma P_0)
 \]
on $\Z[C_N]$. We deduce that 
\[
\delta(\sE_P)=\sum_{\gamma \in \overline{\SL_2(\Z\slash N\Z)}} \overline{F}(\gamma^{-1}P)((\gamma P_{\infty})-(\gamma P_0))
=\sum_{\gamma \in  \overline{\SL_2(\Z\slash N\Z)}} (\overline{F}(\gamma^{-1}P)-\overline{F}(\sigma \gamma^{-1}P))(\gamma P_{\infty}).
\]
We note that $F(\sigma P)=-F(P)$, hence the above sum reduces to 
\[
\delta(\sE_P)=2\sum_{\gamma \in \overline{\SL_2(\Z\slash N\Z)}} \overline{F}(\gamma^{-1}P)(\gamma P_{\infty}).
\]
We calculate the coefficient of $(Q)$ in the above expression using the alternative 
description of $F(P)$ of  Prop. \ref{Alternative}. The  coefficient of $(Q)$ is:
\[
2 \sum_{\gamma P_{\infty}=Q} \sum_{s=(s_1,s_2) \in (\Z/N\Z)^2} e(\frac{s(H \gamma^{-1}P)^0}{2N}) \overline{B_1}(\frac{s_1}{2N})\overline{B_1}(\frac{s_2}{2N}). 
\]
Suppose $\gamma_0 \in \overline{\SL_2(\Z/N\Z)}$ be such that $\gamma_0 P_{\infty} = Q$ (such an element $\gamma_0$ always exist). Consider the matrix 
$T_{\alpha} =\left(\begin{smallmatrix}
1 & \alpha\\
0 & 1\\
\end{smallmatrix}\right)$ with $\alpha \in \Z/N\Z$. We have 
\[
\{\gamma|\gamma \in \SL_2(\Z/N\Z), \gamma P_{\infty}=Q\}=\{\gamma_0 T_{\alpha}| \alpha \in (\Z/N\Z)\}. 
\]
Hence, the coefficient of $(Q)$ is equal to 
\[
2 \sum_{s=(s_1,s_2) \in (\Z/2N\Z)^2} \sum_{\alpha \in (\Z/N\Z)} e(\frac{s(HT_{\alpha} \gamma_0^{-1}P)^0}{2N}) \overline{B_1}(\frac{s_1}{2N})\overline{B_1}(\frac{s_2}{2N}):=T(P, Q).
\]
 We will now prove that:

\begin{equation*}
T(P, Q) =
\begin{cases}
2NF(\mu P_{\infty}) & \text{if $Q=\mu P$,} \\
- \frac{N}{2}& \text{if $Q \neq \mu P$ for all $\mu\in (\Z/N\Z)$}. \\
\end{cases}
\end{equation*}
We examine the quantity 
\[
 \sum_{\alpha \in (\Z/N\Z)} e(\frac{s(HT_{\alpha} \gamma_0^{-1}P)^0}{2N}). 
\]

Let us first assume that $P= \mu Q$  with 
$\mu \in (\Z/N\Z)^*$. We then have $T_{\alpha} \gamma_0^{-1}P = \mu P_{\infty}$ and
\[
T(P,Q)=2 N \sum_{s=(s_1,s_2) \in (\Z/2 N\Z)^2} e(\frac{s(H \mu P_{\infty})^0}{2N}) \overline{B_1}(\frac{s_1}{2N})\overline{B_1}(\frac{s_2}{2N})=2NF(\mu P_{\infty}). 
\]

Assume now  $P  \neq \mu Q$ for all $\mu \in (\Z/N\Z)^*$.  For $\gamma_0^{-1} P
=\begin{pmatrix}u\\v\end{pmatrix}$, we have $v \neq 0 \pmod N$ since $Q\neq \mu P$. 
 
Writing $s=(s_1,s_2)$, it is easy to see  that:
\begin{eqnarray*}
L&=& s(HT_{\alpha} \gamma_0^{-1}P)^0\\
&=&(s_1,s_2) \begin{pmatrix}(u+(\alpha+1)v)^0\\ (-u-(\alpha-1)v)^0\end{pmatrix}\\
&=&s_1 (u +(\alpha + 1)v)^0 + s_2 (-u -(\alpha-1)v)^0.
\end{eqnarray*}
From the above expression, we have:
$L \equiv (s_1 - s_2)u + \alpha(s_1 -s_2)v +(s_2 + s_1)v \pmod N$
 and $L \equiv s_1-s_2 \pmod 2$. Hence, we conclude that:
\[
e(\frac{L}{2N})=e(\frac{L}{2})e(\frac{tL}{N})=e(\frac{s_1-s_2}{2})e(\frac{t[(s_1 - s_2)u + \alpha(s_1 -s_2)v +(s_2 + s_1)v ]}{N})
\]
with $t=\frac{1-N}{2}$. If $(s_1 -s_2)v \not \equiv 0$ modulo $N$, we deduce that 
\[
\sum_{\alpha \in \Z/N\Z} e(\frac{s(HT_{\alpha} \gamma_0^{-1}P)^0}{2N})=0.
\]

We now calculate $T(P,Q)$ by restricting the sum to the terms that satisfy the equality $(s_1 - s_2)v \equiv 0 \pmod {N}$.  
Let us assume $s_1 v \equiv s_2 v \pmod {N}$, then:  
\[
 L \equiv (s_1 - s_2)u  +(s_2 + s_1)v \equiv s_1(u+v)  -s_2(u-v)\pmod {N}.
\]
Let $d={\rm gcd}(v,N)$ and $d'=N/d$. By assumption, one has $d<N$ and consequently $d'>1$.
Observe that $s_1 v \equiv s_2 v \pmod {N}$ is equivalent to $s_1 \equiv s_2 \pmod {d'}$ and that $u$ is well defined modulo $d$.

One has
\begin{eqnarray*}
T(P,Q) & = &    2 N \sum_{s_1 v \equiv s_2v  \pmod {N}}  e(\frac{t[s_1(u+v)-s_2(u-v)]}{N}) \overline{B_1}(\frac{s_1}{2N})\overline{B_1}(\frac{s_2}{2N})e(\frac{s_1-s_2}{2})\\
& = & 
 2 N \sum_{s_1 \in (\Z/2N\Z)} e(\frac{ts_1(u+v)}{N}) \overline{B_1}(\frac{s_1}{2N}) 
  \sum_{s_2  \equiv s_1   \pmod {d'}}  e(-\frac{t s_2(u-v)}{N}))\overline{B_1} (\frac{s_2}{2N})e(\frac{s_1-s_2}{2}).\\
   \end{eqnarray*}
   
Set $s_{2}=s_{1}+kd'$, $\alpha=e(2tv/N)$ (a primitive $d'$-th root of unity),  $\zeta=-e(-tu/d)$
and $\beta=-\zeta$ (a primitive $2d$-th root of unity), where $k$ is an integer modulo $2d$.  The calculation becomes ($s_{1}$ runs through the integers modulo $2N$ and $k$ runs through the integers modulo $2d$)
 \begin{eqnarray*}  
T(P,Q)  &=&  2 N \sum_{s_1} e(\frac{ts_1 
  (u+v)}{N}) \overline{B_1}(\frac{s_1}{2N})
  \sum_{k}  e(-\frac{t(u-v)[s_1+kd']}{N}) \overline{B_1} (\frac{s_1+k d'}{2N})e(-\frac{kd'}{2})\\
    &=&  2 N \sum_{s_1=0}^{2N-1} \alpha^{s_{1}} \overline{B_1}(\frac{s_1}{2N}) 
     \sum_{k=0}^{2d-1}\beta^{k}\overline{B_1} (\frac{s_1}{2N}+\frac{k}{2d}).\\
  \end{eqnarray*}
  
  Set $s_{1}=ad'+b$, with $b$ the class modulo $2N$ of an integer in $[0, d')$ and $a$ well defined modulo $2d$. Hence we get ($a$ and $k$ run through the integers modulo $2d$ and $b$ runs through the integers modulo $d'$)
   
   \begin{eqnarray*}  
T(P,Q)  &=&  2 N \sum_{a=0}^{2d-1} \sum_{b=0}^{d'-1}\alpha^{ad'+b} \overline{B_1}(\frac{b}{2N}+\frac{a}{2d})  \sum_{k}\beta^{k}\overline{B_1} (\frac{b}{2N}+\frac{a}{2d}+\frac{k}{2d})\\
     &=& 2 N \sum_{a=0}^{2d-1} \sum_{b=0}^{d'-1}\alpha^{b} \overline{B_1}(\frac{b}{2N}+\frac{a}{2d}) 
     \sum_{k=0}^{2d-1}\beta^{k}\overline{B_1} (\frac{b}{2N}+\frac{a}{2d}+\frac{k}{2d}).\\
     \end{eqnarray*}
     We need a lemma.
 \begin{lemma}
\label{twistedB}
Let $f(x)=\sum_{k}\beta^{k}\overline{B_1} (x+\frac{k}{2d})$, for $x\in {\R}$. Then
for $l$ integer, $0\le l<2d$ and $x\in(l/2d,(l+1)/2d)$, we have 
 \begin{eqnarray*}  
   f(x)= \sum_{k}\beta^{k}\overline{B_1} (x+\frac{k}{2d})=\frac{\beta^{-l}}{(\beta-1)},
     \end{eqnarray*}
   and      
   \begin{eqnarray*}  
f( \frac{l}{2d})=   \sum_{k}\beta^{k}\overline{B_1} (\frac{l}{2d}+\frac{k}{2d})=\frac{\beta^{-l}(1+\beta)}{2(\beta-1)}.
     \end{eqnarray*}
     \end{lemma}
     \begin{proof}
     Indeed, as a linear combination of functions which are $1$-periodic, affine of slope $1$ on the interval $(l/2d,(l+1)/2d)$, $f$ is $1$-periodic and affine of slope $\sum_{k=0}^{2d-1}\beta^{k}=0$ on $(l/2d,(l+1)/2d)$. Hence $F$ is constant on $(l/2d,(l+1)/2d)$. The relation $f(x+1/2d)=\beta^{-1}f(x)$ follows from the $1$-periodicity of $\overline{B_1}$. The first formula of the lemma follows from the particular case $x\in (0,1/2d)$, where $f(x)=\sum_{k}\beta^{k}k/2d=1/(\beta-1)$. The second formula is obtained by remarking that the values of $f$ at a point of discontinuity $y$ are obtained by averaging its values in the right and left neighborhoods of $y$ (this is a property that $f$ inherits from $\overline{B_1}$ by linearity). 
     \end{proof}
     We can resume the calculation of $T(P,Q)$ using Lemma \ref{twistedB} and the relation $\sum_{b}\alpha^{b}=0$, since $\alpha$ is a primitive $d'$-th root of unity and $d'>1$. We obtain thus (here again $a$ runs through the integers modulo $2d$ and $b$ runs through the integers modulo $d'$):
 \begin{eqnarray*}  
T(P,Q)  
&=&  2 N \sum_{a} \sum_{b}\alpha^{b} \overline{B_1}(\frac{b}{2N}+\frac{a}{2d}) \frac{\beta^{-a}}{\beta-1}+2 N \sum_{a}\overline{B_1} (\frac{a}{2d})\frac{\beta^{-a}(1+\beta)}{2(\beta-1)}\\
 &=& 
  2 N   \sum_{b=1}^{d'-1}\frac{\alpha^{b}}{(\beta-1)(\beta^{-1}-1)}+2 N \frac{(1+\beta)}{2(\beta-1)} \frac{1+\beta^{-1}}{2(\beta^{-1}-1)}\\
   &=& 
  -\frac{2N}{(\beta-1)(\beta^{-1}-1)}+2 N \frac{(1+\beta)}{2(\beta-1)} \frac{1+\beta^{-1}}{2(\beta^{-1}-1)}\\
&=& 
- \frac{2N}{(\beta-1)(\beta^{-1}-1)}[1-  \frac{(1+\beta)(1+\beta^{-1})}{4}]\\
 &=&-\frac{N}{2}.
  \end{eqnarray*}   
\end{proof}    

 \section{The retraction formula}
We now prove Theorem~\ref{Retraction}. Consider the following modular symbol:
\[
\sE(\gamma)=\frac{1}{2N \phi(N)}\sum_{\alpha \in (\Z/N\Z)^*} \sum_{\chi \in D_N^{+}} 
 \frac{{\chi}(\alpha)}{L(\chi)}(\sE_{\alpha \gamma P_{\infty}}-\sE_{\alpha \gamma P_{0}})
\]

(We will see a bit later that $L(\chi)\ne0$.)

The Theorem~\ref{Retraction} follows from the proposition.
\begin{prop}
\label{done}
We have:
\[
\delta(\xi(\gamma))=\delta(\sE(\gamma)).
\]
\end{prop}
\begin{proof}
It is easy to see that:
\[
\delta(\xi(\gamma))=(\gamma P_{\infty})-(\gamma P_0). 
\]
 It is enough to prove that 
\[
(\gamma P_{\infty})=\frac{1}{2N \phi(N)}\sum_{\alpha \in (\Z/N\Z)^*} \sum_{\chi \in D_N^{+}} 
\frac{{\chi}(\alpha)}{L(\chi)} \delta(\sE_{\alpha \gamma P_{\infty}}) + R,
\]
where $R\in\C[C_{N}]$ is a quantity independent of $\gamma$.

We use theorem~\ref{boundary} and proceed using a Fourier inversion formula in $\C [C_{N}]$.
For $\chi\in D_{N}^{+}$ (the group of even Dirichlet characters modulo $N$), denote by $e_{\chi}=(1/\phi(N))\sum_{\beta\in \Z/N\Z)^*}\chi(\beta)[\beta]$ the corresponding idempotent of the group ring $\C[\Z/N\Z)^*]$. Such a group ring acts on $\C[C_{N}]$.
Recall that we have introduced 
\[
L(\chi)=\sum_{\beta}\chi(\beta)(F(\beta P_{\infty})+1/4)=\frac{1}{2}\sum_{\mu\in(\Z/N\Z)^{*}}\frac{\chi(\mu)}{1-\cos(\frac{\pi\mu^{0}}{N})}.
\]

Set $U=\sum_{Q \in C_{N}} (Q)$. We have, using Theorem~\ref{boundary}:

\begin{eqnarray*}  
\sum_{\beta\in (\Z/N\Z)^*}\chi(\beta)\delta(\sE_{\beta \gamma P_{\infty}}) 
&=&  2 N \sum_{\beta\in (\Z/N\Z)^*}\sum_{\mu\in  (\Z/N\Z)^*} \chi(\beta)(F(\mu P_{\infty})+\frac{1}{4})(\mu\beta \gamma P_{\infty})-\frac{N}{2}\sum_{\beta\in(\Z/N\Z)^*}U   \\
 &=& 2N\phi(N)L(\chi)e_{\chi}(\gamma P_{\infty})-\frac{N}{2}\delta_{\chi,1}\phi(N)U.
  \end{eqnarray*}   

Hence we get
\[
e_{\chi}(P)=\frac{1}{2N\phi(N)}\sum_{\beta\in (\Z/N\Z)^*}\frac{\chi(\beta)}{L(\chi)}
\delta(\sE_{\beta \gamma P_{\infty}})+\frac{\delta_{\chi,1}}{4L(\chi)}U.
\]
As $\sum_{\chi\in D_{N}^{+}}e_{\chi}=\frac{1}{2}([1]+[-1])$, we can now sum over $\chi\in D_{N}^{+}$ to obtain
\[
(\gamma P_{\infty})=\frac{1}{2N\phi(N)}\sum_{\beta \in (\Z/N\Z)^*} \sum_{\chi \in D_N^{+}} 
\frac{{\chi}(\alpha)}{L(\chi)} \delta(\sE_{\alpha \gamma P_{\infty}}) +\frac{1}{4L(1)}U.
\]
The term $\frac{1}{4L(1)}U$ is indeed independent of $\gamma$. In other words, we get the same term for $(\gamma P_0)$ also. Hence we have proved the proposition.
\end{proof}

\begin{cor}
The map $\C[C_{N}]\rightarrow \C[E(N)]$ which to $\overline{P}$ associates $\sE_{P}$ is sujective and its kernel is one dimensional, generated by $\sum_{\overline{P}\in C_{N}}[P]$.
\end{cor}
\begin{proof}
Consider the isomorphism of complex vector spaces : 
$E(N)\rightarrow \C[\partial_{N}]^{0}$, which to $\sE$ associates $\delta(\sE)$, where $\C[\partial_{N}]^{0}$ is formed by elements of degree $0$ of $\C[\partial_{N}]$. 

Moreover the map $\delta$ : $ \HH_1(X(N), \partial_N,\C) \rightarrow\C[\partial_{N}]^{0}$ is sujective. Hence by proposition \cite{done}, the map $\overline{P}\mapsto \sE_{P}$ is sujective.

A computation of dimensions tells us that its kernel is of dimension $1$. Moreover, since $\overline{P}\mapsto \sE_{P}$ is a morphisms of $\overline{\SL_2(\Z/N\Z)}$-modules, its kernel is a one dimensional $\overline{\SL_2(\Z/N\Z)}$-module. Since $\overline{\SL_2(\Z/N\Z)}$ acts transitively on $C_{N}$, all coefficients of an element of the kernel are equal. Hence, the kernel is generated by $\sum_{\overline{P}\in C_{N}}[P]$.

\end{proof}

It remains to express $L(\chi)$ in terms of values of Dirichlet $L$-series (and that it is nonzero). Denote by $\chi_{2}$ the Dirichlet character modulo $2N$ which coincides with $\chi$ on odd integers. Let $L(\chi_{2},2)=\sum_{n=1}^{\infty}\chi_{2}(n)n^{-2}$ (which is nonzero).

\begin{prop}
 \label{special}
 One has 
\[
L(\chi)=
\frac{2N^{2}}{\pi^{2}}L(\chi_{2},2).
 \]
\end{prop}
\begin{proof}
Let's remark first that, for $x\in\R$,
\[
\frac{2}{1-\cos(\pi x)}=\frac{1}{\pi^{2}}\sum_{n\in\Z}\frac{1}{(x/2+n)^{2}}=\frac{1}{\pi^{2}}(\sum_{n=1}^{\infty}\frac{1}{(x/2+n)^{2}}+\sum_{n=1}^{\infty}\frac{1}{(x/2+1-n)^{2}})=\frac{1}{\pi^{2}}(\zeta(2,x/2)+\zeta(2,1-x/2)).
 \]
where $\zeta(s,r)=\sum_{n=1}^{\infty}1/(n+r)^{s}$ is the Hurwitz zeta function.
Therefore, we have 
\begin{eqnarray*}
L(\chi)
&=&\frac{1}{2}\sum_{\mu\in (\Z/N\Z)^{*}}\frac{\chi(\mu)}{1-\cos(\pi\mu^{0}/N)}\\
&=&\frac{1}{4\pi^{2}}\sum_{\mu\in (\Z/N\Z)^{*}}\chi(\mu)(\zeta(2,\frac{\mu^{0}}{2N})+\zeta(2,1-\frac{\mu^{0}}{2N}))\\
&=&\frac{1}{2\pi^{2}}\sum_{\mu\in (\Z/N\Z)^{*}}\chi(\mu)\zeta(2,\frac{\mu^{0}}{2N})\\
&=&\frac{1}{2\pi^{2}}\sum_{\alpha=0}^{2N-1}\chi_{2}(\alpha)\zeta(2,\frac{\alpha}{2N})\\
&=&\frac{2N^{2}}{\pi^{2}}L(\chi_{2},2).\\
\end{eqnarray*}
The last equality results from the classical relation between Dirichlet $L$-series and the Hurwitz zeta function [\cite{MR2312338}, Prop. 10.2.5, p. 168]:
\[
L(s, \chi)=\frac{1}{m^s} \sum_{r=0}^{m} \chi(r) \zeta(s, \frac{r}{m}),
\]
where $m$ is the level of $\chi$.
\end{proof}
Given that the cuspidal subgroup of $X(N)$ can be identified with $R(\HH_1(X(N), \partial_N, \Z))/\HH_1(X(N), \Z)$, we recover the classical fact that the order of this cuspidal group is related to the algebraic part of values at $2$ of Dirichlet $L$-series.

\bibliographystyle{Crelle}
\bibliography{Eisensteinquestion.bib}
\end{document}